\documentclass[11pt, a4paper, leqno]{amsart}

\usepackage[utf8]{inputenc}
\usepackage[T1]{fontenc}
\usepackage[english]{babel}
\usepackage{csquotes}

\usepackage{comment}
\usepackage[margin=1in]{geometry}
\usepackage{microtype}         
\usepackage{xcolor}            

\usepackage{soul}

\definecolor{color1}{HTML}{4169e1}
\definecolor{color2}{HTML}{d2d2d2} 

\usepackage{enumitem}          
\setlist[enumerate,1]{label=(\arabic*)} 

\usepackage{graphicx}          
\usepackage{booktabs}          

\usepackage{pgfplots}
\pgfplotsset{compat=1.18} 
\usetikzlibrary{arrows}  

\usepackage{etoolbox}          
\usepackage{chngcntr}          

\usepackage{mathtools}         
\usepackage{amssymb}           
\usepackage{bm}                
\usepackage{mleftright}        
\mleftright

\usepackage{mathrsfs}          
\usepackage{upgreek}           

\usepackage{tikz}              
\usepackage{quiver}            

\newtheorem{theorem}{Theorem}[section]
\newtheorem{lemma}[theorem]{Lemma}
\newtheorem{proposition}[theorem]{Proposition}
\newtheorem{corollary}[theorem]{Corollary}

\theoremstyle{definition}
\newtheorem{definition}[theorem]{Definition}

\newtheorem{remark}[theorem]{Remark} 

\theoremstyle{remark}

\counterwithout{equation}{section}

\mathtoolsset{showonlyrefs,showmanualtags}

\usepackage[colorlinks=true,
            linkcolor=color1,
            citecolor=color1,
            urlcolor=color2,
            pagebackref=true,
            unicode=true]{hyperref}

\usepackage[capitalise, nameinlink]{cleveref}

\crefname{theorem}{Theorem}{Theorems}
\crefname{lemma}{Lemma}{Lemmas}
\crefname{proposition}{Proposition}{Propositions}
\crefname{corollary}{Corollary}{Corollaries}
\crefname{conjecture}{Conjecture}{Conjectures}
\crefname{definition}{Definition}{Definitions}
\crefname{example}{Example}{Examples}
\crefname{problem}{Problem}{Problems}
\crefname{exercise}{Exercise}{Exercises}
\crefname{question}{Question}{Questions}
\crefname{remark}{Remark}{Remarks}
\crefname{notation}{Notation}{Notations}
\crefname{convention}{Convention}{Conventions}

\crefformat{equation}{(#2#1#3)}
\crefmultiformat{equation}{(#2#1#3)}{ and (#2#1#3)}{, (#2#1#3)}{, and (#2#1#3)}
\crefrangeformat{equation}{(#3#1#4)--(#5#2#6)}

\makeatletter
\newcommand\mkcal@one[1]{\expandafter\gdef\csname c#1\endcsname{\ensuremath{\mathcal{#1}}}}
\forcsvlist{\mkcal@one}{A,B,C,D,E,F,G,H,I,J,K,L,M,N,O,P,Q,R,S,T,U,V,W,X,Y,Z}

\newcommand\mksf@one[1]{\expandafter\gdef\csname s#1\endcsname{\ensuremath{\mathsf{#1}}}}
\forcsvlist{\mksf@one}{A,B,C,D,E,F,G,H,I,J,K,L,M,N,O,P,Q,R,S,T,U,V,W,X,Y,Z}

\newcommand\mkbb@one[1]{\expandafter\gdef\csname #1#1\endcsname{\ensuremath{\mathbb{#1}}}}
\forcsvlist{\mkbb@one}{A,B,C,D,E,F,G,H,I,J,K,L,M,N,O,P,Q,R,S,T,U,V,W,X,Y,Z}

\newcommand\mkbf@one[1]{\expandafter\gdef\csname b#1\endcsname{\ensuremath{\mathbf{#1}}}}
\forcsvlist{\mkbf@one}{A,B,C,D,E,F,G,H,I,J,K,L,M,N,O,P,Q,R,S,T,U,V,W,X,Y,Z}

\newcommand\mkrm@one[1]{\expandafter\gdef\csname r#1\endcsname{\ensuremath{\mathrm{#1}}}}
\forcsvlist{\mkrm@one}{A,B,C,D,E,F,G,H,I,J,K,L,M,N,O,P,Q,R,S,T,U,V,W,X,Y,Z}

\newcommand\mkfrak@one[1]{\expandafter\gdef\csname f#1\endcsname{\ensuremath{\mathfrak{#1}}}}
\forcsvlist{\mkfrak@one}{A,B,C,D,E,F,G,H,I,J,K,L,M,N,O,P,Q,R,S,T,U,V,W,X,Y,Z}

\newcommand\mkfraklow@one[1]{\expandafter\gdef\csname fr#1\endcsname{\ensuremath{\mathfrak{#1}}}}
\forcsvlist{\mkfraklow@one}{a,b,c,d,e,f,g,h,i,j,k,l,m,n,o,p,q,r,s,t,u,v,w,x,y,z}
\makeatother

\DeclareMathOperator{\Coh}{Coh}

\DeclareMathOperator{\ch}{ch}

\DeclareMathOperator{\Db}{D^b}

\DeclareMathOperator{\Hom}{Hom}

\DeclareMathOperator{\rk}{rk}

\DeclareMathOperator{\td}{td}

\DeclareMathOperator{\arcosh}{arcosh}

\DeclareMathOperator{\reg}{reg}

\newcommand{\R}{\RR}

\newcommand{\Q}{\QQ}

\newcommand{\K}{\rK}

\title[Stability conditions on threefolds]{Stability conditions on threefolds}
 \author{Yiran Cheng}
 \address{Department of Mathematics, Imperial College London, London SW7 2AZ, United Kingdom
 \newline
 }
 \email{y.cheng@imperial.ac.uk}

 \author{Soheyla Feyzbakhsh}
 \email{s.feyzbakhsh@imperial.ac.uk}

\begin{document}
\maketitle
\vspace{-.8 cm}
\begin{abstract}
We investigate a subspace of Bridgeland stability conditions on $\PP^n$ satisfying the so-called Li condition. These are the stability conditions whose restriction to a smooth projective subvariety $X \subset \PP^n$ is again a stability condition. We then show that, when $X$ is a threefold, the restricted stability conditions coincide with those obtained via the double-tilt construction introduced by Bayer--Macr\`i--Toda. As an application, we prove the weak BMT conjecture.
\end{abstract}

\section{Introduction}
 

Recently, Chunyi Li \cite{Li26} proved that the Bridgeland stability conditions on products of elliptic curves constructed by Yucheng Liu \cite{Liu21} descend through the quotient by the natural action of $(\mathbb{Z}/2\mathbb{Z})^n\rtimes S_n$, building on a result of Polishchuk \cite{Pol07}. As a consequence, this yields a family of Bridgeland stability conditions
\[
\sigma^{\PP^n}_{a,b}=(\cA^{\PP^n}_{a,b},Z^{\PP^n}_{a,b}),
\qquad
(a,b)\in\mathbb{Q}_{>0}\times\mathbb{Q},
\]
on the projective space $\PP^n$. Their central charge is given by the physical central charge
\[
Z_{a,b}^{\PP^n}(E)
=
-\int_{\PP^n} e^{-(b+ia)H_{\PP^n}}\,\ch(E),
\]
where $H_{\PP^n}=c_1(\mathcal{O}_{\PP^n}(1))$. Our first result shows that, for sufficiently large $a$, these stability conditions satisfy the so-called Li condition, which guarantees that they restrict to Bridgeland stability conditions on smooth projective subvarieties of $\PP^n$.

\begin{theorem}[= Corollary~\ref{cor:large-volume}]\label{thm-intro-1}
Let $m\in\ZZ_{>0}$. If $a>\frac{mn}{\uppi}$, then every nonzero object $E\in\Db(\PP^n)$ satisfies
\begin{equation}\label{eq:Pn-phase}
    \phi_{a,b}^{-}(E\otimes\cO_{\PP^n}(m))
    <
    \phi_{a,b}^{-}(E)+1,
\end{equation}
where $\phi_{a,b}^{-}(E)$ denotes the minimum phase of the Harder--Narasimhan filtration of $E$ with respect to the stability condition $\sigma_{a,b}$.
\end{theorem}

Now let $\iota \colon X\hookrightarrow\PP^n$ be a smooth projective subvariety. From a resolution of $\iota_*\cO_X$ by split vector bundles, we define a rational number 
\[
m(X\xhookrightarrow{\iota}\PP^n) \in \mathbb{Q},
\]
see Definition~\ref{def:m}. This invariant is bounded above by the Castelnuovo--Mumford regularity of the ideal sheaf $\cI_{X/\PP^n}$. Combining \cite[Theorem~6.5]{Li26} with Theorem~\ref{thm-intro-1}, we deduce that if $ a>\frac{n}{\uppi}\,m(X\xhookrightarrow{\iota}\PP^n),
$
then
\[
\sigma^X_{a,b}\coloneqq\iota^{\sharp}\sigma^{\PP^n}_{a,b}
\]
is a Bridgeland stability condition on $X$ with heart
\[
\cA^X_{a,b}
=
\{E\in\Db(X):\iota_*E\in\cA^{\PP^n}_{a,b}\},
\]
and central charge
\begin{equation}\label{physical-central}
Z^X_{a,b}(E)
=
-\int_X
e^{-(b+ia)H_X}
\td(N_{X/\PP^n})^{-1}
\ch(E).
\end{equation}

On the other hand, there is the tilt construction introduced by Bridgeland \cite{Bri08} and further developed by Bayer--Macr\`i--Toda \cite{Bayer-Macri-Toda:2014}. Given the heart $\cA$ of a bounded $t$-structure together with a weak stability function $Z\colon \K_{\mathrm{num}}(\cA)\to\mathbb{C}$, namely a group homomorphism sending every nonzero object of $\cA$ to the upper half-plane or the non-positive real axis, one obtains a torsion pair on $\cA$ by separating stable objects according to the sign of the real part of $Z$. Tilting with respect to this torsion pair produces a new heart, denoted by $\cA^Z$.

Let $(X,H)$ be a smooth polarized threefold. It was conjectured in \cite{Bayer-Macri-Toda:2014, Bayer-Macri-Stellari:2016, Bayer-Macri:0} that the heart obtained by the double-tilt construction on $\Coh(X)$, together with the central charge \eqref{physical-central}, defines a Bridgeland stability condition on $X$. Our next theorem confirms this prediction for sufficiently large values of $a$. To state it, we introduce the invariant
\begin{equation}
    a(X,H)
    \coloneqq
    \min \left\{
    \frac{1}{2}n\,m(X\xhookrightarrow{\iota}\PP^n)
    :
    \text{$\iota$ is a closed immersion with }
    H=c_1\bigl(\iota^*\cO_{\PP^n}(1)\bigr)
    \right\}.
\end{equation}

\begin{theorem}[= Theorem~\ref{thm-main-section3}] \label{thm-intro-2}
Let $(X,H)$ be a smooth projective threefold, and let $(a,b)\in\mathbb{Q}^2$ satisfy $a>a(X,H)$. Then the heart $\cA^X_{a,b}$ of the induced stability condition coincides with the double-tilted heart
\[
\bigl(\Coh(X)^{\mathcal Z_b^{(1)}}\bigr)^{\mathcal Z_{a,b}^{(2)}}
\]
where
\begin{equation}
    \mathcal{Z}^{(1)}_b
    =
    -H^2\smash{\widetilde\ch}^{bH}_1
    +iH^3\smash{\widetilde\ch}^{bH}_0,
    \qquad
    \mathcal{Z}^{(2)}_{a,b}
    =
    -H\smash{\widetilde\ch}^{bH}_2
    +\frac{a^2}{6}H^3\smash{\widetilde\ch}^{bH}_0
    +iH^2\smash{\widetilde\ch}^{bH}_1,
\end{equation}
and $ \smash{\widetilde\ch}^{bH}
=
e^{-bH}\,\td(N_{X/\PP^n})^{-1}\ch$. 
\end{theorem}
As an immediate consequence of Theorem~\ref{thm-intro-2}, we obtain a twisted version of the weak Bayer--Macr\`i--Toda conjecture\footnote{To deduce higher-rank Donaldson--Thomas theory from the rank-zero case as in \cite{Feyzbakhsh-Thomas:2024}, one needs the strong Bayer--Macr\`i--Toda conjecture \cite[Conjecture~1.3.1]{Bayer-Macri-Toda:2014}, which differs from \eqref{ch3} by replacing the coefficient of $a^2$ from $\frac{1}{2}$ to $\frac{1}{18}$ and the strict inequality by a non-strict one.} \cite[Conjecture~3.2.7]{Bayer-Macri-Toda:2014}.

\begin{corollary}[= Corollary~\ref{Cor.BMT}]
In the setting of Theorem~\ref{thm-intro-2}, let
$E\in\Coh(X)^{\mathcal Z_b^{(1)}}$
be $\mathcal Z_{a,b}^{(2)}$-semistable and satisfy
\[
H\smash{\widetilde\ch}_2^{bH}(E)
=
\frac{a^2}{6}H^3\smash{\widetilde\ch}_0^{bH}(E).
\]
Then 
\begin{equation}\label{ch3}
\smash{\widetilde\ch}_3^{bH}(E)
<
\frac{a^2}{2}H^2\smash{\widetilde\ch}_1^{bH}(E).   
\end{equation}
\end{corollary}

A remarkable consequence of the works of Polishchuk \cite{Pol07}, Liu \cite{Liu21}, and Li \cite{Li26} is a reversal of the traditional philosophy behind proving Castelnuovo-type inequalities such as the BMT conjecture. Previously, one first exploited the geometry of the underlying variety to establish such inequalities and then deduced the existence of Bridgeland stability conditions. In contrast, these works provide a purely categorical construction of Bridgeland stability conditions on $\Db(X)$, from which Castelnuovo-type inequalities follow as consequences. 

\subsection*{Acknowledgements} We are grateful to Arend Bayer, Chunyi Li, Yucheng Liu, Zhiyu Liu, and Richard Thomas for many helpful discussions. Both authors were supported by the Royal Society through the University Research Fellowship URF/R1/231191.

\bigskip

\section{stability conditions on smooth projective varieties}\label{section 1}
In this section, we begin by studying a two-dimensional slice of the Bridgeland stability conditions on products of elliptic curves constructed by Liu \cite{Liu21}, obtaining an explicit upper bound for the distance between two stability conditions (Theorem~\ref{thm:distance}). We then recall Li's descent of these stability conditions to $\PP^n$ \cite{Li26}. Our main result in this section is Theorem~\ref{thm-intro-1} from the Introduction (proved as Corollary~\ref{cor:large-volume}), which identifies a family of stability conditions on $\PP^n$ satisfying the phase condition required in Li's restriction theorem for inducing Bridgeland stability conditions on smooth projective subvarieties of~$\PP^n$. Throughout the paper, all varieties are assumed to be defined over the field of complex numbers $\CC$, and all functors between derived categories are derived.


\subsection{Stability conditions on $C^n$} The story begins with Yucheng Liu’s construction of stability conditions on products of curves, which was subsequently further developed by several authors. The following theorem summarizes the statement that we will need in this paper.

\begin{theorem}[\cite{Liu21,FLZ22,LMPSZ}]\label{symmetric stability}
Let $C$ be an elliptic curve, and let $\cO_C(1)$ be an ample line bundle of degree one on $C$.
Then, for any $n \geq 1$, there exists a continuous family of stability conditions
\begin{equation}
    \sigma_{a,b}^{(n)} \coloneqq (\cP_{a,b}^{(n)}, Z_{a,b}^{(n)})
\end{equation}
on $\Db(C^n)$, parametrized by $a \in \Q_{>0}$ and $b \in \Q$, such that:
\begin{itemize}
    \item the central charge is given by
    \begin{equation}
        Z_{a,b}^{(n)}(-)
        =
        -\int_{C^n} e^{-(b+ia)\widetilde H}  \ch(-),
    \end{equation}
    where
    \[
        \widetilde H
        \coloneqq
        c_1\!\left(
        \cO_C(1)\boxtimes \cdots \boxtimes \cO_C(1)
        \right);
    \]
    
    \item all skyscraper sheaves are stable of phase $1$;
    
    \item every stability condition $\sigma_{a,b}$ is $(\ZZ/2\ZZ)^n \rtimes \fS_n$-invariant, where $(\ZZ/2\ZZ)^n \rtimes \fS_n$ acts on $C^n$ via factorwise involutions and permutations.
\end{itemize}
\end{theorem}



Let
\[
q_i \colon C^n \to C
\]
be the projection onto the $i$-th factor, and define
\[
H_i \coloneqq c_1\!\left(q_i^*\cO_C(x)\right)\in N^1(C^n),
\]
where $x\in C$ is any point. The divisor class $H_i$ is represented by
\[
q_i^*(x)
=
C^{i-1}\times\{x\}\times C^{n-i},
\]
which, by abuse of notation, we denote simply by $C^{n-1}$ whenever the choice of factor and point is irrelevant.

By construction, the central charge $Z_{a,b}^{(n)}$ factors through the map
\[
v_n \colon \K(C^n) \to \Lambda_n \simeq \ZZ^{\oplus 2^n},
\]
where
\begin{equation}
    v_n(E)
    =
    \left(
    H_{j_1}\cdots H_{j_\ell}\,\ch_{n-\ell}(E)
    \right)_{1 \leq j_1 < \cdots < j_\ell \leq n}.
\end{equation}
Moreover, the stability conditions $\sigma_{a,b}^{(n)}$ satisfy the support property with respect to the lattice $\Lambda_n$; see \cite[Section~4.2]{LMPSZ}.

As a consequence of the construction of stability conditions on products of curves, we obtain the following mild restriction lemma.

\begin{lemma} \label{lem:restriction}
  For any $\sigma_{a,b}^{(n)}$-semistable object $E \in \Db(C^n)$, we have
\begin{equation}\label{eq of lem}
    a\, \bigl| Z^{(n-1)}_{a,b}(E|_{C^{n-1}}) \bigr|
    \leq
    \bigl| Z^{(n)}_{a,b}(E) \bigr| ,
\end{equation}
where we identify $C^n = C^{n-1}\times C$. 
\end{lemma}
\begin{proof}
    In the proof, we fix $(a,b)$ and suppress these indices from the notation. We denote the central charge of the stability condition $\sigma_{a,b}^{(n-1)}$ on $C^{n-1}$ (resp.\ $\sigma_{a,b}^{(n)}$ on $C^n$) by $Z^{(n-1)}$ (resp.\ $Z^{(n)}$). We denote the two projections by 
\[
p \colon C^n \to C^{n-1},
\qquad
q_n \colon C^n \to C.
\]

Following \cite[Section 4]{Liu21}, for any object $E \in \Db(C^{n-1}\times C)$, we write
\begin{equation}\label{expression}
    Z^{(n-1)}\!\left(
    p_{*}\bigl(E \otimes q_n^*\cO_C(k)\bigr)
    \right)
    =
    A(E)k + B(E) + i\bigl(C(E)k + D(E)\bigr),
\end{equation}
where
\[
A,B,C,D \colon K(C^n)\to \mathbb{R}
\]
are group homomorphisms.
    We have
    \begin{equation}
        v_{n-1}(E|_{C^{n-1}}) = v_{n-1}(p_*E ) - v_{n-1}(p_*(E(-H_n))).
    \end{equation}
    In particular, we get 
    \begin{align}
         Z^{(n-1)}(E|_{C^{n-1}}) = &   Z^{(n-1)}(p_*E) -  Z^{(n-1)}(p_*(E(-H_n)))\\
         = & A(E) +i C(E) .
    \end{align}
    By \cite[Lemma 4.4]{LMPSZ}, the central charge on $C^n$ takes the form
    \begin{equation}\label{central}
        Z^{(n)}(E) = -(b+ia)(A(E) +iC(E)) +(B(E) + iD(E)) ,
    \end{equation}
    so we know 
    \begin{align*}
        B(E) C(E) - A(E) D(E) = \,& \Im \left( \overline{Z^{(n)}(E)} \, Z^{(n-1)}(E|_{C^{n-1}}) \right) - a |Z^{(n-1)}(E|_{C^{n-1}})|^2 \\
        \leq \, & |Z^{(n)}(E)\, Z^{(n-1)}(E|_{C^{n-1}})| - a |Z^{(n-1)}(E|_{C^{n-1}})|^2.
    \end{align*}
    Applying \cite[Lemma~5.7]{Liu21} with $\eta=0$ implies that\footnote{In \cite{Liu21} there is no parameter for this $b$, but this does not change anything in the proofs.} $B(E) C(E) -A(E) D(E) \geq 0$ for any $\sigma_{a,b}^{(n)}$-semistable object $E$ and so the claim follows. 
\end{proof}

\begin{remark}
  The inequality \eqref{eq of lem} in Lemma~\ref{lem:restriction} is in fact sharp, and equality can be achieved by $\sigma^{(n)}_{a,b}$-semistable objects. Indeed, let $F$ be a slope-semistable vector bundle on $C$ of rank $r$ and degree $d$ satisfying
\[
\frac{d}{r}=b,
\]
which always exists; see, for example, \cite{Ati57}. Let 
\[
i \colon C \hookrightarrow C^n,\qquad
i(y)=(z_1,\ldots,z_{n-1},y),
\]
where $z_1,\ldots,z_{n-1}\in C$ are fixed points. By \cite[Theorem~1.7]{cheng2025bridgeland}, the pushforward $i_*F$ is stable with respect to any Bridgeland stability condition on $C^n$. Since $p\circ i$ is the constant map to the point
\[
z\coloneqq (z_1,\ldots,z_{n-1})\in C^{n-1},
\]
we have
\[
p_*\bigl(i_*F\otimes q_n^*\mathcal O_C(k)\bigr)
\simeq
\mathbf{R}\Gamma \left(C,F\otimes\mathcal O_C(k)\right)\otimes\mathcal O_z.
\]
Therefore,
\[
Z_{a, b}^{(n-1)}
\bigl(
p_*(i_*F\otimes q_n^*\mathcal O_C(k))
\bigr)
=
-(kr+d),
\]
and hence, using \eqref{expression},
\[
A(i_*F)=-r,\qquad
B(i_*F)=-d,\qquad
C(i_*F)=D(i_*F)=0.
\]
Substituting these into \eqref{central}, we obtain
\begin{equation}\label{eqq.1}
Z^{(n)}(i_*F)
=
r(b+ia)-d
=
iar.
\end{equation}
Now let
\[
j\colon C^{n-1}\hookrightarrow C^n,\qquad
j(x_1,\ldots,x_{n-1})
=
(x_1,\ldots,x_{n-1},z_n),
\]
where $z_n\in C$ is fixed. Then
\[
j^*i_*F
\simeq
F_{z_n}\otimes\mathcal O_{z},
\]
and therefore
\begin{equation}\label{eqq.2}
Z^{(n-1)}\!\left(j^*i_*F\right)
=
-r.
\end{equation}
Comparing \eqref{eqq.1} and \eqref{eqq.2} shows that equality holds in \eqref{eq of lem} for $E=i_*F$.

\end{remark}


In Li's construction \cite{li_real_2025} of stability conditions on smooth projective varieties, a key step is to control the phase change of objects under tensor products. To achieve this, we use Bridgeland's generalised metric on the stability manifold \cite[Proposition~8.1]{Bri07}. Let $\sigma_0,\sigma_1$ be two stability conditions on a triangulated category $\cD$. 
Consider the normalized generalized metric (which differs slightly from Bridgeland's original definition) defined by
\begin{equation}\label{eq:dist}
    d(\sigma_0,\sigma_1)
    =
    \sup_{0\neq E\in\cD}
    \left\{
    |\phi_1^+(E)-\phi_0^+(E)|,
    \,
    |\phi_1^-(E)-\phi_0^-(E)|,
    \,
    \frac{1}{\uppi} \left|
    \log \frac{m_1(E)}{m_0(E)}
    \right|
    \right\}.
\end{equation}
Here, $m_i(E)=\sum_j | Z(E_i^j) |$ denotes the mass of $E$ with respect to $\sigma_i$, where $E_i^j$ are its HN factors. The main result of this section is the following theorem.

\begin{theorem}\label{thm:distance}
    For any $\Delta b \geq 0$, we have
    \begin{equation}
        d( \sigma_{a,b}^{(n)}, \sigma^{(n)}_{a,b+ \Delta b} ) \leq \frac{n}{\uppi} \frac{\Delta b}{a}.
    \end{equation}
\end{theorem}

To prove the theorem, we first state the following general lemma to control the distance.

\begin{lemma}\label{lem:finsler}
Let $\sigma_t = (\cP_t, Z_t)$ be a smooth path of stability conditions for $t \in [0,1]$ on a triangulated category $\cD$. For any $t\in [0,1]$, define
\[
\lambda(t) \coloneqq \sup \left\{ 
\frac{1}{\uppi}\left| \frac{ \frac{d}{dt}Z_t(E)}{Z_t(E)} \right| 
: 0 \neq E \text{ is } \sigma_t\text{-semistable}
\right\}.
\]
Then
\[
d(\sigma_0, \sigma_1) \le \int_0^1 \lambda(t)\, dt.
\]
\end{lemma}
\begin{proof}
    We write $Z=|Z|e^{i \uppi \phi}$, then for any $\sigma_t$-semistable object $E$, we have 
  \begin{align}\label{eq:hypotenuse}
     \frac{\frac{d}{dt}Z_t(E)}{Z_t(E)} =  \frac{d}{dt} \log Z_t(E) =\, & \frac{d}{dt} \log |Z_t(E)| + i \uppi \frac{d}{dt} \phi_t(E) \\
     = \, &   \frac{\frac{d}{dt}|Z_t(E)|}{|Z_t(E)|} + i \uppi \frac{d}{dt} \phi_t(E) .
     \end{align}

  For each nonzero object $F \in \cD$, choose $ 0=t_0<t_1<\cdots<t_k=$
such that the Harder--Narasimhan filtration of $F$ with respect to $\sigma_t$ remains constant for all $ t\in (t_{i-1},t_i)$. We denote by 
\[
E_i^1, E_i^2, \dots,E_i^{\ell_i} 
\]
the Harder--Narasimhan factors of $F$ with respect to $\sigma_t$ for $t\in (t_{i-1},t_i)$, and by $E_i^+$ (resp.\ $E_i^-$) the factor of maximal (resp.\ minimal) phase.
    Then, for each $i$, we have
    \begin{equation} 
        |\phi^+_{t_{i}}(F) - \phi^+_{t_{i-1}}(F) | = |\phi_{t_{i}}(E_i^+) - \phi_{t_{i-1}}(E_i^+) | \leq \int_{t_{i-1}}^{t_{i}} \left|\frac{d}{dt} \phi_t(E_i^+)\right| dt \overset{\eqref{eq:hypotenuse}}{\leq} \frac{1}{\uppi} \int_{t_{i-1}}^{t_{i}} \left|\frac{\frac{d}{dt}Z_t(E_i^+)}{Z_t(E_i^+)}\right|  dt .
    \end{equation}
    Hence
    \begin{equation} \label{eq:phi+}
        | \phi^+_1(F) - \phi^+_0(F) |  \leq \sum_{i=1}^k |\phi^+_{t_{i}}(F) - \phi^+_{t_{i-1}}(F)| \leq  \int_0^1 \lambda(t)\, dt .
    \end{equation}
    The same argument applies to $\phi^-(F)$. 
    Similarly, for the mass we have
    \begin{equation} 
        \log \frac{m_{t_{i}}(F)}{m_{t_{i-1}}(F)} = \int_{t_{i-1}}^{t_{i}} \frac{d}{dt} \log m_t(F) dt =\int_{t_{i-1}}^{t_{i}} \frac{\frac{d}{dt} m_t(F)}{m_t(F)} dt 
    \end{equation}
    and
    \begin{equation}
        \left|\frac{d}{dt} m_t(F)\right| \leq \sum_{j} \frac{d}{dt} |Z_t(E^j_i)| \overset{\eqref{eq:hypotenuse}}{\leq} \sum_{j} \left| Z_t(E^j_i) \, \frac{\frac{d}{dt}Z_t(E^j_i)}{Z_t(E^j_i)} \right| \leq \uppi \, m_t(F) \lambda(t)
        .
    \end{equation}
    Therefore, we have
    \begin{equation} \label{eq:mass}
        \left| \log \frac{m_1(F)}{m_0(F)} \right| \leq \sum_{i=1}^k \left| \log \frac{m_{t_{i}}(F)}{m_{t_{i-1}}(F)} \right| \leq \uppi \int_{0}^{1} 
        \lambda(t)\,dt
    \end{equation}
    which completes the proof. 
\end{proof}

\begin{proof}[Proof of Theorem~\ref{thm:distance}]
    
    Consider the horizontal line in the $(b,a)$-plane.
    A direct computation yields
  \begin{align}\label{eq:partial}
    \frac{\partial}{\partial b} Z_{a,b}^{(n)}(E)
    &=
    \int_{C^n}
    \widetilde H \, e^{-(b+ia)\widetilde H} \ch(E)
    \\
    &=
    n\int_{C^n}
    H_n \, e^{-(b+ia)(\widetilde H-H_n)}
     \ch(E)
    \\
    &=
    n\int_{C^{n-1}}
    e^{-(b+ia)(\widetilde H-H_n)}
     \ch(E|_{C^{n-1}})
    \\
    &=
    -\,n\, Z_{a,b}^{(n-1)}(E|_{C^{n-1}}),
\end{align}
where we identify $C^n = C^{n-1}\times C$.     
    So, for all $t \in [0,\Delta b]$ and all $\sigma_{a,b+t}$-semistable objects $E \in \Db(C^n)$, the Lemma \ref{lem:restriction} implies
    \begin{equation}
        \left| \frac{\frac{\partial}{\partial b}Z_{a,b}^{(n)}(E)}{Z_{a,b}^{(n)}(E)} \right| = \left| \frac{n Z_{a,b}^{(n-1)}(E|_{C^{n-1}})}{Z_{a,b}^{(n)}(E)} \right| \leq \frac{n}{a} .
    \end{equation}
    Hence, the claim follows from  Lemma \ref{lem:finsler}.
\end{proof}

\begin{remark} \label{rem:geodesic}
    Via a similar computation as in \eqref{eq:partial}, one gets
    \begin{equation}
        \frac{\partial}{\partial a} Z_{a,b}^{(n)}(E) = - i n Z_{a,b}^{(n-1)}( E|_{C^{n-1}} ) .
    \end{equation}
    Now let $(b(t),\,a(t))$ for $t \in [0, 1]$ be a smooth path in the upper half-plane such that starting at $(b, a)$ and ending at $(b+ \Delta b , a+ \Delta a)$ for $\Delta a, \Delta b \geq 0$. Then
\[
\frac{d}{dt}Z_{a(t),\,b(t)}^{(n)}(E)
=
-n
Z_{a(t),\,b(t)}^{(n-1)}(E|_{C^{n-1}})
\left(
\frac{db(t)}{dt}
+
i\frac{da(t)}{dt}
\right).
\]
Hence by Lemma~\ref{lem:restriction} we get
\[
\left|
\frac{\frac{d}{dt}Z_{a(t),\,b(t)}^{(n)}(E)}
{Z_{a(t),\,b(t)}^{(n)}(E)}
\right|
\leq
\frac{n}{a(t)}
\sqrt{
\left(\frac{db(t)}{dt}\right)^2
+
\left(\frac{da(t)}{dt}\right)^2
}\,.
\]
That is to say, the Bridgeland metric is bounded by some multiple of the hyperbolic metric of Poincar\'{e} half-plane model.
If we take the path to be the geodesic (that is, the semicircle connecting  $(b, a)$ and $(b+ \Delta b , a+ \Delta a)$),
then Lemma \ref{lem:finsler} implies
    \begin{equation}
        d( \sigma_{a,b}, \sigma_{a+\Delta a, b+ \Delta b} ) \leq \frac{n}{\uppi} \arcosh \left( 1+ \frac{(\Delta b)^2+(\Delta a)^2}{2a (a+ \Delta a)} \right)
    \end{equation}
where $\arcosh (x) = \log(x + \sqrt{x^2-1})$ for $x \geq 1$.
    In particular, if $\Delta a =0 $, we have
    \begin{equation}
        d( \sigma_{a,b}, \sigma_{a,b+ \Delta b} ) \leq \frac{n}{\uppi} \arcosh \left( 1+ \frac{(\Delta b)^2}{2a^2} \right) 
    \end{equation}
    which slightly improves the bound in Theorem \ref{thm:distance}.

\end{remark}

\subsection{Geometric stability conditions on $\PP^n$ and its subvarieties}
The next step in Li's construction is to show that stability conditions on $C^n$, where $C$ is an elliptic curve as before, induce stability conditions on $\PP^n$. We begin with a general definition. 
\begin{definition}
Let $f \colon Y \to X$ be a finite morphism between smooth projective varieties. Given a pre-stability condition $\sigma_Y=(\cP,Z)$
on $\Db(Y)$, we define its \emph{pushforward} to $\Db(X)$ by
\[
f_\sharp\sigma_Y
\coloneqq
(f_\sharp\cP,f_\sharp Z),
\]
where
\[
f_\sharp\cP(\phi)
\coloneqq
\{
E\in\Db(X)
\mid
f^*E\in\cP(\phi)
\},
\]
for every $\phi\in\RR$, and
\[
f_\sharp Z
\coloneqq
Z\circ f^*
\colon
\K_0(X)
\xrightarrow{\,f^*\,}
\K_0(Y)
\longrightarrow
\CC.
\]
Conversely, given a pre-stability condition $\sigma_X=(\cP,Z)$ on $\Db(X)$, we define its \emph{pullback} to $\Db(Y)$ by
\[
f^\sharp\sigma_X
\coloneqq
(f^\sharp\cP,f^\sharp Z),
\]
where
\[
f^\sharp\cP(\phi)
\coloneqq
\{
E\in\Db(Y)
\mid
f_*E\in\cP(\phi)
\},
\]
for every $\phi\in\RR$, and
\[
f^\sharp Z
\coloneqq
Z\circ f_*
\colon
\K_0(Y)
\xrightarrow{\,f_*\,}
\K_0(X)
\longrightarrow
\CC.
\]
\end{definition}

\medskip

Recall that the group $
(\ZZ/2\ZZ)^n \rtimes \fS_n$
acts on $C^n$ via factorwise involutions and permutations. Let
\[
\pi \colon C^n \to \PP^n \simeq C^n /\bigl((\ZZ/2\ZZ)^n \rtimes \fS_n\bigr)
\]
be the quotient map.
\begin{theorem}[\cite{Li26}]\label{thm:stab P^n}
For any $(\ZZ/2\ZZ)^n \rtimes \fS_n$-invariant stability condition $\sigma=(\cP,Z)$ on $C^n$, the pushforward $\pi_\sharp\sigma$ defines a stability condition on $\PP^n$ satisfying
\begin{equation}\label{ineq:P^n}
    \phi^{\pm}_{\pi_\sharp\sigma}(E)
    \leq
    \phi^{\pm}_{\pi_\sharp\sigma}(E\otimes\cO_{\PP^n}(1))
\end{equation}
for every nonzero object $E\in\Db(\PP^n)$.
\end{theorem}

\begin{proof}
 By the proof of \cite[Theorem~6.3]{Li26}, we know that $\pi_\sharp\sigma$ defines a stability condition on $\PP^n$. It therefore suffices to show \eqref{ineq:P^n}, which is equivalent to
\begin{equation}\label{ineq:C^n}
    \phi_{\sigma}^{\pm}(\pi^*E)
    \leq
    \phi_{\sigma}^{\pm}(\pi^*E \otimes \cO_{C^n}(2\widetilde H)),
\end{equation}
for every nonzero object $E\in\Db(\PP^n)$, where $\widetilde H=H_1+\cdots+H_n$. The inequality \eqref{ineq:C^n} follows from \cite[Lemma~2.3]{Li26}.
\end{proof}

As s result of Theorem \ref{thm:stab P^n}, we have a two-parameter family of stability conditions on $\PP^n$ as 
$$
\sigma_{a,b}^{\PP^n} \coloneqq \pi_\sharp  \sigma^{(n)}_{2a,\, 2b}
$$
for $(a,b)\in \Q_{>0} \times \Q$ where the central charge is given by 
\begin{equation}
    Z^{\PP^n}_{a, b}(-)=-\int_{\PP^n} e^{-(b+i a) H_{\PP^n}}  \ch (-) 
\end{equation}
for $H_{\PP^n}=c_1(\cO_{\PP^n}(1))$. 

As an application of Theorem~\ref{thm:distance}, we obtain the following corollary, which will play a key role in the construction of stability conditions on subvarieties. For an object $E \in \Db(\PP^n)$, we denote by $\phi_{a,b}^{+}(E)$ and $\phi_{a,b}^{-}(E)$ its maximal and minimal phases, respectively, with respect to the stability condition $\sigma_{a,b}^{\PP^n}$.

\begin{corollary}\label{cor:large-volume}
For any $m \in \ZZ_{>0}$ and any $a > \frac{mn}{\uppi}$, every nonzero object $E \in \Db(\PP^n)$ satisfies
\begin{equation}\label{eq:Pn-phase}
    \phi_{a,b}^{-}(E \otimes \cO_{\PP^n}(m))
    <
    \phi_{a,b}^{-}(E)+1.
\end{equation}
\end{corollary}
\begin{proof}
    By construction, an object $E \in \Db(\PP^n)$ is $\sigma^{\PP^n}_{a,b}$-semistable of phase $\phi$ if and only if its pullback $\pi^*E$ is $\sigma^{(n)}_{2a,2b}$-semistable of phase $\phi$. Therefore,
    \begin{equation}
        d(\sigma^{\PP^n}_{a,b}, \sigma^{\PP^n}_{a,b+\Delta b})
        \leq
        d(\sigma^{(n)}_{2a,2b}, \sigma^{(n)}_{2a,2b+2\Delta b})
        \leq
        \frac{n}{\uppi} \frac{\Delta b}{a},
    \end{equation}
    where the last inequality follows from Theorem~\ref{thm:distance}. Taking $\Delta b=m$ and using $a> \frac{mn}{\uppi}$, we obtain
    \begin{equation}\label{eq}
        d(\sigma^{\PP^n}_{a,b}, \sigma^{\PP^n}_{a,b+m}) <1.
    \end{equation}

On the other hand, we have
\[
\sigma^{(n)}_{2a,2b}\otimes \cO_{C^n}(2m\widetilde H)
=
\sigma^{(n)}_{2a,2(b+m)}.
\]
Indeed, the equality of the corresponding central charges follows from a direct computation, while the equality of the stability conditions then follows from \cite[Theorem~1.1]{LMPSZ}.
   Applying $\pi_\sharp$ to both sides and using \cite[Lemma~3.2]{Li26}, we deduce that
    \[
    \sigma^{\PP^n}_{a,b}\otimes \cO_{\PP^n}(m)
    =
    \sigma^{\PP^n}_{a,b+m}.
    \]
    Therefore, \eqref{eq} implies
    \begin{equation}
        |\phi^-_{a,b}(E\otimes \cO_{\PP^n}(-m))
        -
        \phi^-_{a,b}(E)|
        =
        |\phi^-_{a,b+m}(E)-\phi^-_{a,b}(E)|
        <1.
    \end{equation}
    This completes the proof.
\end{proof}




\begin{definition}\label{def:m}
Let
\[
\iota\colon X\hookrightarrow\PP^n
\]
be a closed subvariety. Consider an exact sequence
\begin{equation}\label{eq:resolution}
    0
    \to
    F
    \to
    \bigoplus_{j=1}^{r_{n-1}}
    \cO_{\PP^n}(-d_{n-1,j})
    \to
    \cdots
    \to
    \bigoplus_{j=1}^{r_1}
    \cO_{\PP^n}(-d_{1,j})
    \to
    \cO_{\PP^n}
    \to
    \iota_*\cO_X
    \to
    0,
\end{equation}
where $F\in\Coh(\PP^n)$ and all the intermediate terms are split vector bundles. We define
\begin{equation}\label{eq:def m}
m(X\xhookrightarrow{\iota}\PP^n)
\coloneqq
\min
\left\{
\max_{\substack{1\le i\le n-1\\1\le j\le r_i}}
\frac{d_{i,j}}{i}
\right\},
\end{equation}
where the minimum is taken over all exact sequences of the form \eqref{eq:resolution}.
\end{definition}


\begin{theorem} \label{thm:effective restriction}
   For any $a, b \in \QQ$ with $a\geq \frac{1}{\uppi} \, n \,  m(X \xhookrightarrow{\iota} \PP^n)$, the pullback 
    \begin{equation}
        \sigma_{a,b}^{X} \coloneqq \iota^{\sharp} \sigma_{a,b}^{\PP^n} 
    \end{equation}
    defines a stability condition on $X$.
\end{theorem}

\begin{proof}
    This follows from Theorem~\ref{thm:distance} and the proof of \cite[Theorem~6.5]{Li26}; we include a proof here for completeness.

By \cite[Corollary~2.2.2]{Pol07}, to restrict a stability condition to $X$, it suffices to prove that
\begin{equation}\label{goal}
    \phi_{a,b}^-(E\otimes \iota_*\cO_X)
    \geq
    \phi_{a,b}^-(E)
\end{equation}
for all $E\in \Db(\PP^n)$.

Consider a resolution of $\iota_*\cO_X$ as in \eqref{eq:resolution} realizing the minimum $m(X \xhookrightarrow{\iota} \PP^n)$. Tensoring this resolution with $E$, we see that $E\otimes \iota_*\cO_X$ admits a filtration in $\Db(\PP^n)$ whose factors are $E$, the objects
\[
E\otimes \cO_{\PP^n}(-d_{i,j})[i],
\qquad
1\leq i\leq n-1,\quad 1\leq j\leq r_i,
\]
and $E\otimes F[n]$. Hence
\begin{equation}\label{bound by filtration}
    \phi_{a,b}^-(E\otimes \iota_*\cO_X)
    \geq
    \min\!\left\{
    \phi_{a,b}^-(E),\,
    \phi_{a,b}^-\bigl(E\otimes \cO_{\PP^n}(-d_{i,j})[i]\bigr),\,
    \phi_{a,b}^-(E\otimes F[n])
    \right\}_{\substack{1\leq i\leq n-1\\1\leq j\leq r_i}}.
\end{equation}

Let $m = m(X \xhookrightarrow{\iota} \PP^n)$. Then, by Theorem~\ref{thm:stab P^n},
\begin{equation}\label{in.1}
    \phi_{a,b}^-(E\otimes \cO_{\PP^n}(-d_{i,j})[i])
    \geq
    \phi_{a,b}^-(E\otimes \cO_{\PP^n}(-im)[i]).
\end{equation}
Since $a\geq \frac{mn}{\uppi}$, applying Corollary~\ref{cor:large-volume} $i$ times yields
\begin{equation}\label{in.1.5}
    \phi_{a,b}^-(E\otimes \cO_{\PP^n}(-im)[i])
    \geq
    \phi_{a,b}^-(E).
\end{equation}
Moreover, since all skyscraper sheaves belong to $\cP_{a,b}^{\PP^n}(1)$, \cite[Lemma~10.1]{Bri08} implies that
\begin{equation}
    E \in \Coh(\PP^n)[\lceil \phi_{a,b}^-(E)\rceil-1,\,+\infty)
    \qquad\text{and}\qquad
    \Coh(\PP^n)\subset \cP^{\PP^n}_{a,b}(1-n,1].
\end{equation}
Hence
\begin{equation}\label{ineq:2}
    E\otimes F[n]
    \in
    \Coh(\PP^n)[\lceil \phi_{a,b}^-(E)\rceil-1+n,\,+\infty)
    \subset
    \cP^{\PP^n}_{a,b}(\lceil \phi_{a,b}^-(E)\rceil,\,+\infty).
\end{equation}
Combining \eqref{in.1}, \eqref{in.1.5}, and \eqref{ineq:2}, the inequality \eqref{bound by filtration} implies \eqref{goal}, completing the proof.
\end{proof}

Note that the lower bound for $a$ in Theorem~\ref{thm:effective restriction} can be slightly improved to
\[
a \geq \frac{m(X \xhookrightarrow{\iota} \PP^n)}{2\pi \sinh(\frac{1}{2n})}
\]
by Remark~\ref{rem:geodesic}.

\begin{remark}\label{rem:m}
    We conclude with a few remarks concerning the invariant $m(X \xhookrightarrow{\iota} \PP^n)$.
    \begin{enumerate}[label=(\roman*)]

        \item In many examples, such as complete intersections, the maximum
        \begin{equation}
            \max_{\substack{1\leq i\leq n-1\\ 1\leq j\leq r_i}}
            \left\{
            \frac{d_{i,j}}{i}
            \right\}
        \end{equation}
        is attained for $i=1$. It would be interesting to know whether this holds in general. 
        
        \item It is easy to check that
        \[
        m(X \xhookrightarrow{\iota} \PP^n)
        \leq
        \reg(\cI_{X/\PP^n}),
        \]
        where $\reg(-)$ denotes Castelnuovo--Mumford regularity.
        
        \item It is also straightforward to verify that for any hyperplane section $Y \subset X$, one has
        \[
        m(Y \xhookrightarrow{\iota|_Y} \PP^{n-1})
        \leq
        m(X \xhookrightarrow{\iota} \PP^n).
        \]
    \end{enumerate}
\end{remark}

\bigskip

\section{Dimensional reduction}

Let $(X,H)$ be a polarized smooth projective variety with $H$ very ample. We define
\begin{equation}
    a(X,H)
    \coloneqq
    \min \left\{
    \frac{1}{2} n \, m(X \xhookrightarrow{\iota} \PP^n)
    :
    \text{$\iota$ is a closed immersion such that $H=c_1(\iota^*\cO_{\PP^n}(1))$}
    \right\}.
\end{equation}
We fix a closed immersion $\iota \colon X \hookrightarrow \PP^n$ realizing the minimum $a(X,H)$, and $a,b \in \Q$ satisfying
\begin{equation}\label{condition.eq}
a \geq a(X,H).    
\end{equation}
As described in Section~\ref{section 1}, there is an induced stability condition
\[
\sigma_{a,b}^X
=
(\cP_{a,b}^X,Z_{a,b}^X),
\]
whose central charge is given (up to a scalar multiple) by
\[
Z_{a,b}^X(E)
=
-\int_X
e^{-(b+ia)H}
\td(N_{X/\PP^n})^{-1}
\ch(E)
.
\]
The slicing is defined by
\[
\cP_{a,b}^X(\phi)
=
\{
E\in \Db(X)
:
\iota_*E\in \cP_{a,b}^{\PP^n}(\phi)
\}.
\]
We denote the heart of $\sigma_{a,b}^X$ by $\cA_{a,b}^X$. In this section, we give an explicit description of this heart when $X$ is a surface or a threefold.

\medskip

To simplify the notation, we write
\begin{equation}
    \smash{\widetilde\ch}^{bH} = e^{-bH}  \td(N_{X/\PP^n})^{-1} \ch .
\end{equation}
In particular, we have the following explicit expression:
\begin{align}
\smash{\widetilde\ch}^{bH}_0 &= \ch_0 = \rk \\
\smash{\widetilde\ch}^{bH}_1 &= \ch^{B_{\!X}}_1 \\
\smash{\widetilde\ch}^{bH}_2 &= \ch^{B_{\!X}}_2 + \frac{1}{12}\ch_2(N_{X/\PP^n}) \ch_0 \\
\smash{\widetilde\ch}^{bH}_3 &= \ch^{B_{\!X}}_3 + \frac{1}{12}\ch_2(N_{X/\PP^n}) \ch^{B_{\!X}}_1
\end{align}
where $\ch^{B_{\!X}} = e^{-B_{\!X}}  \ch$ and
\begin{equation}\label{BX}
    B_{\!X} \coloneqq  bH + \frac{c_1(N_{X/\PP^n})}{2}.
\end{equation}
To study the heart $\cA^X_{a,b}$, we consider the following slope function on coherent sheaves:
\[
\mu_{B_{\!X}}(E)
\coloneqq
\frac{H^{\dim X-1} \ch_1^{B_{\!X}}(E)}
     {H^{\dim X}\ch_0(E)}
=
\frac{H^{\dim X-1} \ch_1(E)}
     {H^{\dim X}\ch_0(E)}
-b
-\frac{H^{\dim X-1}c_1(N_{X/\PP^n})}
       {2H^{\dim X}}
\]
whenever $\ch_0(E)\neq 0$, and $\mu_{B_{\!X}}(E)=+\infty$ otherwise. 
The tilted heart $\Coh^{B_{\!X}}(X)$ is the extension-closure 
\begin{equation}\label{tilted heart}
\Coh^{B_{\!X}}(X) = \langle \mathcal{T}_{B_{\!X}}, \mathcal{F}_{B_{\!X}}[1] \rangle    
\end{equation}
where $\mathcal{T}_{B_{\!X}}$ (resp. $\mathcal{F}_{B_{\!X}}$) consists of sheaves $E \in \Coh(X)$ with $\mu^-_{B_{\!X}}(E) >0$ (resp. $\mu^+_{B_{\!X}}(E) \leq 0$)\footnote{This tilted heart is denoted by $\Coh(X)^{\mathcal{Z}^{(1)}_b}$ in the Introduction, where
\[
\mathcal{Z}^{(1)}_b
=
H^{\dim X-1}\bigl(-\smash{\widetilde\ch}^{\,bH}_1
+iH\,\smash{\widetilde\ch}^{\,bH}_0\bigr).
\] }. 

\medskip

Let $Y \in |H|$ be a smooth hyperplane section, and denote its embedding by 
$$\jmath\colon Y \hookrightarrow X.$$ 
Let $H_{Y} \coloneqq \jmath^* H$. Note that by Remark~\ref{rem:m} we automatically have $a \geq \frac{n}{\uppi} \, m(Y,H_{Y})$,
so we have the induced stability condition
$\sigma^Y_{a, b} = (\cA_{a,b}^Y , Z_{a,b}^Y)$ on $Y$ as well. 

For simplicity, we write $\phi_{a,b}^{\pm}(E)$ for the maximal and minimal phases of an object $E$ with respect to the relevant stability condition; depending on whether $E$ belongs to $\Db(Y)$, $\Db(X)$, or $\Db(\PP^n)$, these phases are computed with respect to $\sigma_{a,b}^Y$, $\sigma_{a,b}^X$, or $\sigma_{a,b}^{\PP^n}$, respectively. 



\begin{lemma} \label{restriction}
  For any object $E \in \Db(X)$, we have
\begin{equation}
    \phi^-_{a,b}(E)
    \leq
    \phi^-_{a,b}(E|_Y)
    \qquad\text{and}\qquad
    \phi^+_{a,b}(E(H)|_Y)
    \leq
    \phi^+_{a,b}(E)+1.
\end{equation}
\end{lemma}
\begin{proof}
Consider the section $\cO_X \overset{s}{\hookrightarrow} \cO_X(H)
\twoheadrightarrow
\jmath_*\cO_Y(H)$. Tensoring with $E$ yields the exact triangle
\begin{equation}\label{exact.1}
    E(H)
    \to
    \jmath_*E(H)|_Y
    \to
    E[1].
\end{equation}
By construction, $\jmath_*\cP^Y_{a,b}(\phi)
\subset
\cP^X_{a,b}(\phi)$, therefore, semistable objects remain semistable under $\jmath_*$, hence the HN filtration is preserved and we get
\begin{equation}\label{eq.1}
    \phi_{a,b}^+(E(H)|_Y) =
    \phi_{a,b}^+\bigl(\jmath_*E(H)|_Y\bigr).
\end{equation}
On the other hand, \eqref{exact.1} implies that
\[
\phi_{a,b}^+\bigl(\jmath_*E(H)|_Y\bigr)
\leq
\max
\bigl\{
\phi_{a,b}^+(E(H)),
\phi_{a,b}^+(E[1])
\bigr\} 
\leq
\phi_{a,b}^+(E)+1,
\]
where the last inequality follows from Corollary~\ref{cor:large-volume} and \cite[Lemma~2.2]{Li26}. Combining this with \eqref{eq.1} proves the second inequality. A similar argument shows that
\[
\phi_{a,b}^-(E|_Y)
=
\phi_{a,b}^-\bigl(\jmath_*E|_Y\bigr)
\geq
\min
\bigl\{
\phi_{a,b}^-(E),
\phi_{a,b}^-(E(-H)[1])
\bigr\} \geq
\phi_{a,b}^-(E).
\]
as claimed in the first inequality. 
\end{proof}

\subsection{Surfaces} When $X$ is a smooth surface $S$, the corresponding heart admits a particularly simple description via Bridgeland's classification of geometric stability conditions on surfaces.
We know
\begin{align}
    Z_{a, b}^S(E) 
    =-\ch_2^{B_{\!S}}(E) + \left( \frac{a^2 H_{\!S}^2}{2} - \frac{\ch_2(N_{S/\PP^n})}{12}   \right) \ch_0(E) \ + i {a} H_{\!S}  \ch_1^{B_{\!S}}(E) .
\end{align}

\begin{lemma}\label{induced heart on surface}
    We have $\cA_{a, b}^S = \Coh^{B_{\!S}}(S)$ where $\Coh^{B_{\!S}}(S)$ is the tilted heart defined in \eqref{tilted heart}. 
\end{lemma}
\begin{proof}
Since the skyscraper sheaf $\cO_x$ for any point $x \in S$ is stable of phase one, \cite[Lemma~10.1]{Bri08} implies that\footnote{The proof of \cite[Lemma~10.1]{Bri08} is valid for any smooth surface.}
\[
\Coh(S) \subset \cP^S_{a, b}(-1, 1].
\]
Hence, we have the torsion pair
\[
\mathcal{T}^S_{a, b} = \Coh(S) \cap \cP^S_{a, b}(0, 1] \ , \qquad  
\mathcal{F}^S_{a, b} = \Coh(S) \cap \cP^S_{a, b}(-1, 0]
\]
on $\Coh(S)$, and
\[
\cP^S_{a, b}(0, 1] = \langle \mathcal{T}^S_{a, b}, \mathcal{F}^S_{a, b}[1] \rangle.
\]
To prove the claim, it suffices to show that
\begin{equation}\label{claim}
    \cT_{B_{\!S}} \subset \mathcal{T}^S_{a, b}
    \qquad\text{and}\qquad
    \cF_{B_{\!S}} \subset \mathcal{F}^S_{a, b}.
\end{equation}
Indeed, since both pairs define torsion pairs, these inclusions imply that the two torsion pairs coincide. For any coherent sheaf $E$, we consider the short exact sequence
\[
0 \to E_1 \to E \to E_2 \to 0
\]
such that $E_1 \in \mathcal{T}^S_{a, b}$ and $E_2 \in \mathcal{F}^S_{a, b}$. 

First assume $E \in \cT_{B_{\!S}}$, i.e.\ $\mu_{B_{\!S}}^-(E) > 0$. If $E_2 \neq 0$, then $\Im (Z_{a, b}(E_2) ) \leq 0$, and hence $\mu_{B_{\!S}}(E_2) \leq 0$, which is impossible. Thus $E = E_1 \in \mathcal{T}^S_{a, b}$.

Similarly, if $E \in \cF_{B_{\!S}}$, i.e.\ $\mu_{B_{\!S}}^+(E) \leq 0$, then either $E_1 = 0$, or $E_1$ is a torsion-free sheaf satisfying
\[
\mu_{B_{\!S}}(E_1) = \mu_{B_{\!S}}^+(E_1) = 0.
\]
However $E_1 \in \cP_{a, b}^S(0, 1]$ with zero imaginary part, hence it is $\sigma_{a, b}^S$-semistable of phase one, which is impossible by \cite[Lemma~10.1.(b)]{Bri08}. This completes the proof of \eqref{claim}. 
\end{proof}



\subsection{Threefolds}
Let $(X,H)$ be a smooth polarized threefold. In this section, we prove Theorem~\ref{thm-intro-2} stated in the Introduction. We know 
\begin{align}
    Z^X_{a, b}(E) 
    &= - \left( \smash{\widetilde\ch}_3^{bH} - \frac{a^2}{2} H^2 \smash{\widetilde\ch}_1^{bH} \right) + i  \left( a H \smash{\widetilde\ch}_2^{bH} - \frac{a^3}{6} H^3 \smash{\widetilde\ch}_0^{bH} \right)
\end{align}
and so
\begin{align}
    \Re  \left( Z^X_{a, b}(E) \right) &= - \ch_3^{B_{\!X}}(E)  + \left( \frac{a^2 H^2}{2}  - \frac{\ch_2(N_{X/\PP^n})}{12}  \right) \ch_1^{B_{\!X}}(E) \nonumber \\
    \frac{1}{a} \Im \left( Z^X_{a, b}(E) \right) &=   H \ch_2^{B_{\!X}}(E)  - \left( \frac{a^2H^2}{6}  - \frac{\ch_2(N_{X/\PP^n})}{12}\right) H  \ch_0(E). \label{imag}
\end{align}

\begin{proposition}\label{prop.tilt}
       For any object $E \in \cP^X_{a, b}( 0,1]$, we have 
   \begin{equation}\label{comp.slope}
       \mu_{B_{\!X}}^+\left(\cH^{-2}(E) \right) \leq -\frac{1}{2} \ , \qquad \mu_{B_{\!X}}^-\left(\cH^0(E) \right) > \frac{1}{2}.  
   \end{equation}
   In particular, we have 
   \begin{equation}\label{first heart}
       \Coh^{B_{\!X}}(X) \subset \cP_{a, b}^X(-1,1]. 
   \end{equation}
\end{proposition}
\begin{proof}
    By the proof of \cite[Lemma~10.1]{Bri08}, any object $E \in  \cP_{a, b}^X(0,1]$ is isomorphic to a length three complex of locally-free sheaves. In particular,
\begin{equation}\label{eq.heart}
     \cP_{a, b}^X(0,1] \subset \langle \Coh(X) , \Coh(X) [1] , \Coh(X)[2] \rangle.
\end{equation}
and $\cH^{-2}(E)$ is a torsion-free sheaf. Let $S \in |H|$ be a smooth hyperplane section. Denote the embedding by $\jmath\colon S \hookrightarrow X$. We take the section $s \colon \cO_X(-H) \hookrightarrow \cO_X$ corresponding to $S$ and tensor by $E$ to get the distinguished triangle
    \begin{equation}
        E(-H) \xrightarrow{s} E \to \jmath_*E|_S  \to E[1].
    \end{equation}
    Taking cohomology with respect to $\Coh(X)$ gives the commutative diagram  
    \begin{equation}\label{diag}
        \begin{tikzcd}
        	0  & {\cH^{-2}(E(-H))} & {\cH^{-2}(E)} & {\cH^{-2}(\jmath_*E|_S)} & \cdots \\
        	0 & {\cH^{-2}(E)(-H)} & {\cH^{-2}(E)} & {\jmath_*\cH^{-2}(E)|_S} & 0.
        	\arrow[from=1-1, to=1-2]
        	\arrow[from=1-2, to=1-3]
        	\arrow[equals, from=1-2, to=2-2]
        	\arrow[from=1-3, to=1-4]
        	\arrow[equals, from=1-3, to=2-3]
        	\arrow[from=1-4, to=1-5]
        	\arrow[from=2-1, to=2-2]
        	\arrow[from=2-2, to=2-3]
        	\arrow[from=2-3, to=2-4]
        	\arrow[hook, from=2-4, to=1-4, "d"]
        	\arrow[from=2-4, to=2-5]
        \end{tikzcd}
    \end{equation}
    On the other hand, for any torsion-free sheaf $F$, the normal bundle sequence
\[
0 \to N_{S/X} \to N_{S/\PP^n} \to N_{X/\PP^n}|_S \to 0
\]
implies
\begin{equation}\label{slope-rest}
    \mu_{B_{\!X}}(F)
    =
    \mu_{B_{\!S}}(F|_S)+\frac{1}{2}.
\end{equation}
Hence
\[
\mu_{B_{\!X}}^+(\cH^{-2}(E))
\overset{(a)}{\leq}
\mu_{B_{\!S}}^+(\cH^{-2}(E)|_S)+\frac{1}{2}
\overset{(b)}{\leq}
\mu_{B_{\!S}}^+\bigl(\cH^{-2}(E|_S)\bigr)+\frac{1}{2}
=
\mu_{B_{\!S}}^+\bigl(\cH^{-2}(E(H)|_S)\bigr)-\frac{1}{2}
\overset{(c)}{\leq}
-\frac{1}{2},
\]
where (a) follows from the torsion-freeness of $\cH^{-2}(E)$ together with \eqref{slope-rest}, and (b) follows from the injectivity of the morphism $d$ in diagram \eqref{diag}. To prove (c), Lemma~\ref{restriction} implies that
\[
E(H)|_S \in \cP^S_{a,b}(0,2].
\]
Therefore, by Lemma~\ref{induced heart on surface}, we have $\cH^{-2}(E(H)|_S)\in \cF_{B_{\!S}}$ which implies (c). This completes the proof of the first inequality in \eqref{comp.slope}.

\medskip

    Similarly, we have the commutative diagram 
    \begin{equation}
        \begin{tikzcd}
        	\dots & {\cH^{0}(E(-H))} & {\cH^{0}(E)} & {\cH^{0}(\jmath_*E|_S)} & 0 \\
        	\dots & {\cH^{0}(E)(-H)} & {\cH^{0}(E)} & {\jmath_*\big(\cH^{0}(E)\big)|_S} & 0
        	\arrow[from=1-1, to=1-2]
        	\arrow[from=1-2, to=1-3]
        	\arrow[equals, from=1-2, to=2-2]
        	\arrow[from=1-3, to=1-4]
        	\arrow[equals, from=1-3, to=2-3]
        	\arrow[from=1-4, to=1-5]
        	\arrow[from=2-1, to=2-2]
        	\arrow[from=2-2, to=2-3]
        	\arrow[from=2-3, to=2-4]
        	\arrow[equals, from=2-4, to=1-4]
        	\arrow[from=2-4, to=2-5]
        \end{tikzcd}
    \end{equation}
    where the surjectivity forces the vertical isomorphism on the right. If $\mu_{B_{\!X}}^-\left(\cH^0(E)\right)= + \infty$ then the claimed inequality trivially holds; so we may assume the slope of the last HN factor is finite, then this factor is torsion-free hence we can argue as before.
   Therefore,
\begin{equation}
    \mu_{B_{\!X}}^-\bigl(\cH^{0}(E)\bigr)
    \geq
    \mu_{B_{\!S}}^-\bigl(\cH^{0}(E)|_S\bigr)+\frac{1}{2}
    =
    \mu_{B_{\!S}}^-\bigl(\cH^{0}(E|_S)\bigr)+\frac{1}{2}
    >
    \frac{1}{2},
\end{equation}
where the last inequality follows from Lemma~\ref{restriction} and Lemma~\ref{induced heart on surface}. This proves the second inequality in \eqref{comp.slope}.

To prove \eqref{first heart}, let $E \in \Coh^{B_{\!X}}(X)$, and denote by $E^+$ (resp.~$E^-$) the HN factor of $E$ with respect to $\sigma^X_{a,b}$ of maximal (resp.~minimal) phase. First assume that $E\in\mathcal{T}_{B_{\!X}}$. If $E^+\in\cP^X_{a,b}(1,\infty)$, then by \eqref{eq.heart} we have $\cH^i(E^+)=0$ for $i\geq 0$, and hence $\Hom(E^+,E)=0$, a contradiction. A similar argument shows that $\phi_{a,b}(E^-)>-2$. Now suppose that $E^-\in\cP^X_{a,b}(-2,-1]$. Then $\mu_{B_{\!X}}^+(\cH^0(E^-))\leq 0$ by \eqref{comp.slope}, while $\cH^i(E^-)=0$ for $i<0$ by \eqref{eq.heart}. It follows that $\Hom(E,E^-)=0$, again a contradiction. Therefore $E^\pm$, and hence $E$, lie in $\cP^X_{a,b}(-1,1]$.

Now suppose that $E\in\mathcal{F}_{B_{\!X}}[1]$. If $E^-\in\cP^X_{a,b}(-\infty,-1]$, then by \eqref{eq.heart} we have $\cH^i(E^-)=0$ for $i<0$, and therefore $\Hom(E,E^-)=0$, which is impossible. A similar argument shows that $\phi_{a,b}(E^+)\leq 2$. If $E^+\in\cP^X_{a,b}(1,2]$, then \eqref{comp.slope} implies that $\mu_{B_{\!X}}^-(\cH^{-1}(E^+))>0$, which contradicts the existence of a nonzero morphism $E^+\to E$, since $\mu_{B_{\!X}}^+(E[-1])\leq 0$. This completes the proof of \eqref{first heart}.

\end{proof}


As a direct corollary of Proposition \ref{prop.tilt}, the pair 
    \begin{equation}\label{first pair}
        \cT^X_{a, b} = \Coh^{B_{\!X}}(X) \cap \cP^X_{a, b}(0, 1] \qquad \text{and} \qquad  
\cF^X_{a, b} = \Coh^{B_{\!X}}(X) \cap \cP^X_{a, b}(-1, 0]
    \end{equation}
defines a torsion pair on $\Coh^{B_{\!X}}(X)$. 

\medskip

On the other hand, one can construct a double-tilted heart as in \cite{Bayer-Macri-Toda:2014}. Assuming that
\begin{equation}\label{eq:weak}
    a^2 \geq \frac{H \ch_2(N_{X/\PP^n})}{2H^3},
\end{equation}
we define, for objects $E\in \Coh^{B_{\!X}}(X)$, the slope function
\begin{equation}\label{slope}
    \nu_{a,b}(E)
    =
    \frac{H \smash{\widetilde\ch}^{bH}_2(E)-\frac{a^2}{6}H^3\smash{\widetilde\ch}^{bH}_0(E)}
    {H^2 \smash{\widetilde\ch}^{bH}_1(E)\vphantom{\widetilde{\ch}}}
    =
    \frac{H \ch_2^{B_{\!X}}(E)-\left(\frac{a^2}{6}H^3-\frac{1}{12}H \ch_2(N_{X/\PP^n})\right)\ch_0(E)}
    {H^2 \ch_1^{B_{\!X}}(E)}
\end{equation}
whenever the denominator is nonzero, and $\nu_{a,b}(E)=+\infty$ otherwise. This defines a torsion pair 
\begin{equation}\label{second pair}
({\cT}'_{a, b},\,{\cF}'_{a, b})    
\end{equation}
on $\Coh^{B_{\!X}}(X)$, where ${\cT}'_{a, b}$ (resp.~${\cF}'_{a, b}$) consists of objects $E\in\Coh^{B_{\!X}}(X)$ satisfying $\nu^-_{a,b}(E)>0$ (resp.~$\nu^+_{a,b}(E)\leq 0$)\footnote{Tilting $\Coh^{B_{\!X}}(X)$ with respect to the torsion pair $({\cT}'_{a,b},{\cF}'_{a,b})$ is precisely the double-tilt construction
\[
\bigl(\Coh(X)^{\mathcal Z_b^{(1)}}\bigr)^{\mathcal Z_{a,b}^{(2)}},
\]
introduced in the Introduction, where
\[
\mathcal Z_b^{(1)}
=
-H^2\smash{\widetilde\ch}^{bH}_1
+iH^3\smash{\widetilde\ch}^{bH}_0,
\qquad
\mathcal Z_{a,b}^{(2)}
=
-H\smash{\widetilde\ch}^{bH}_2
+\frac{a^2}{6}H^3\smash{\widetilde\ch}^{bH}_0
+iH^2\smash{\widetilde\ch}^{bH}_1.
\]}. 
The following elementary lemma shows that condition \eqref{eq:weak} is automatically satisfied under assumption \eqref{condition.eq}.



\begin{lemma}\label{lem:bound}
    We have
    \begin{equation}
        \frac{ H  \ch_2(N_{X/\PP^n})}{2 H^3} 
        < \left(\frac{m(X \xhookrightarrow{\iota} \PP^n)\,n}{2}\right)^2 .
    \end{equation}
\end{lemma}

\begin{proof}
Combining the Euler sequence $0 \to \mathcal{O}_X \to \mathcal{O}_X(H)^{\oplus(n+1)} \to T_{\PP^n}|_X \to 0$ and the normal bundle sequence $0 \to T_X \to T_{\PP^n}|_X \to N_{X/\PP^n} \to 0$ shows that $N_{X/\PP^n}$ is globally generated. In particular, $N_{X/\PP^n}$ is nef, and therefore
\begin{equation}\label{eq:bound c2}
    H c_2(N_{X/\PP^n}) \geq 0.
\end{equation}

On the other hand, by taking $m=\max\{d_{1,j}\} \leq m(X \xhookrightarrow{\iota} \PP^n)$ from the resolution \eqref{eq:resolution},
the sheaf $\mathcal{I}_{X/\PP^n}(m)$ is globally generated. 
Hence the surjection
\begin{equation}
    \mathcal{I}_{X/\PP^n}(m)\twoheadrightarrow
    \mathcal{I}_{X/\PP^n}/\mathcal{I}_{X/\PP^n}^2\otimes\mathcal{O}_X(m)
    =
    N_{X/\PP^n}^{\vee}(m)
\end{equation}
implies
\begin{equation}\label{eq:bound c1}
    0 \leq c_1\!\bigl(N_{X/\PP^n}^{\vee}(m)\bigr) H^2
    =
    -c_1(N_{X/\PP^n}) H^2 + m(n-3)H^3.
\end{equation}
Therefore,
\begin{align}
    \frac{H \ch_2(N_{X/\PP^n})}{H^3}
    &=
    \frac{H c_1(N_{X/\PP^n})^2-2H c_2(N_{X/\PP^n})}{2H^3}
    \overset{\eqref{eq:bound c2}}{\leq}
    \frac{H c_1(N_{X/\PP^n})^2}{2H^3}
    \overset{(*)}{\leq}
    \frac{1}{2}\left(\frac{H^2 c_1(N_{X/\PP^n})}{H^3}\right)^2 \\
    &\overset{\eqref{eq:bound c1}}{<}
    \frac{(mn)^2}{2},
\end{align}
where $(*)$ follows from the Hodge Index Theorem.
\end{proof}


Finally, we show that the two torsion pairs defined in \eqref{first pair} and \eqref{second pair} coincide, thereby completing the proof of our main theorem.

\begin{theorem}\label{thm-main-section3}
Let $(X,H)$ be a smooth projective threefold, and let $a,b\in\mathbb{Q}$ satisfy
\[
a> a(X, H).
\]
    Then the heart of the induced stability condition $\sigma^X_{a,b}$ is given by
    \begin{equation}
        \cP^X_{a,b}(0,1]
        =
        \langle {\cT}'_{a, b},\,{\cF}'_{a, b}[1]\rangle,
    \end{equation}
    where $({\cT}'_{a, b},\,{\cF}'_{a, b})$ is the torsion pair described in \eqref{second pair}.
\end{theorem}
\begin{proof}
    Similar to Lemma \ref{induced heart on surface}, we only need to show
\begin{equation}
    {\cT}'_{a, b} \subset \cT^X_{a, b} \qquad \text{and} \qquad 
    {\cF}'_{a, b} \subset \cF^X_{a, b}.
\end{equation}
For any object $E \in \Coh^{B_{\!X}}(X)$, there is a short exact sequence
\[
0 \to E_1 \to E \to E_2 \to 0
\]
in $\Coh^{B_{\!X}}(X)$ such that $E_1 \in \cT^X_{a, b}$ and $E_2 \in \cF^X_{a, b}$. 

First assume $E \in {\cT}'_{a, b}$, i.e.\ $\nu_{a, b}^-(E) > 0$. If $E_2 \neq 0$, then $\Im ( Z^X_{a, b}(E_2) ) \leq 0$. Then comparing the equation of \eqref{imag} with the slope \eqref{slope} implies $\nu_{a, b}(E_2) \leq 0$ , which is impossible. Thus $E = E_1 \in \cT_{a, b}$.

Similarly, let $E \in {\cF}'_{a,b}$, so that $\nu_{a,b}^+(E)\leq 0$. Then either $E_1=0$, or $E_1\in\cP^X_{a,b}(0,1]$ has zero imaginary part. In the latter case, $E_1$ is $\sigma_{a,b}^X$-semistable of phase one and satisfies $\nu_{a,b}(E_1)=0$. However, the next Lemma \ref{lem-phase one} implies that $E_1$ is a zero-dimensional sheaf, and hence $\nu_{a,b}(E_1)=+\infty$, a contradiction.
\end{proof}

\begin{lemma}\label{lem-phase one}
    Any object $E \in \Coh^{B_X}(X) \cap \cP_{a, b}^X(1)$ is a sheaf supported in dimension zero.
\end{lemma}
\begin{proof}
    First assume $E$ is $\sigma^X_{a, b}$-stable of phase one, 
    Since the skyscraper sheaf $\cO_x$ for every point $x \in X$ is $\sigma^X_{a, b}$-stable of phase one, any $\sigma^X_{a, b}$-stable object does not have any morphism to $\cO_x$, and so its $\cH^0$ vanishes. This implies that for our $\sigma^X_{a, b}$- semistable object $E$ its $\cH^0(E)$ is supported in finitely many points on $X$. Taking a smooth section $S \in |H|$ avoiding these points, we have $\cH^0(E_1)|_S=0$, hence $E|_S[-1] \eqcolon F$ is a sheaf. Now assume $F \neq 0$, then by Lemma~\ref{restriction}, we know
    \begin{equation}
       1= \phi_{a, b}^-(E) \leq \phi_{a, b}^-(E|_S) \leq \phi_{a, b}^+(E(H)|_S) \leq \phi_{a, b}^+(E)+1=2
    \end{equation}
    hence $F \in \cP^S_{a,b}[0,1]$.
    In particular, we know every of its HN factors satisfies $\Im Z^S_{a,b} \geq 0$, hence $\Im Z^S_{a,b}(F) \geq 0$. Therefore, we have
    \begin{equation}
        \mu_{B_S}(F) \geq 0,
    \end{equation}
    and so
    \begin{equation}
        \mu_{B_X}(\cH^{-1}(E)) = \mu_{B_S}(F) + \frac{1}{2} \geq \frac{1}{2},
    \end{equation}
    which is not possible as $E \in \Coh^{B_X}(X)$.  
\end{proof}

\begin{corollary} \label{Cor.BMT}
    In the setup of Theorem \ref{thm-main-section3},
    any $\nu_{a,b}$-semistable object $E \in \Coh^{B_{\!X}}(X)$ with $\nu_{a,b}(E)=0$, i.e., 
    $$
     H \smash{\widetilde\ch}_2^{bH} = \frac{a^2}{6} H^3 \smash{\widetilde\ch}_0^{bH} 
    $$
satisfies 
\begin{equation}
    \smash{\widetilde\ch}_3^{bH} <  \frac{a^2}{2} H^2 \smash{\widetilde\ch}_1^{bH}.
\end{equation}
\end{corollary}

\subsection{BMT(S)-type inequality} In this subsection, we assume that Theorem~\ref{symmetric stability} holds for all $(a,b)\in\mathbb{R}_{>0}\times\mathbb{R}$. This result will be established in the forthcoming paper \cite{LLL+}. Under this hypothesis, all the remaining results of the present paper continue to hold for real values of $(a,b)$ satisfying
\[
a\geq a(X,H).
\]
The goal in this subsection is to apply ideas similar to those of \cite{BMS16} to obtain a Bogomolov-type inequality for all tilt-semistable objects. 


\begin{theorem} \label{BMS}
    Let $(X,H)$ be a smooth projective threefold. If $E \in \Coh^{B_X}(X)$ is $\nu_{a, b}$-semistable for some $(a, b) \in \R^2$ with $
a\geq a(X, H)$, then 
    \begin{equation}\label{eq:BMS}
        \frac{1}{2} a^2 e_1^4 + 2 e_1^2 \left(e_2 - \frac{a^2}{6} e_0\right)\left(e_2 - \frac{a^2}{3} e_0\right) - \frac{4}{3} e_0 \left(e_2 - \frac{a^2}{6} e_0\right)^3 -  e_1^3 e_3 \geq 0,
    \end{equation}
    where $e_i \coloneqq H^{3-i}  \smash{\widetilde\ch}^{bH}_i (E)$. 
\end{theorem}

\medskip

Before proceeding with the proof, we introduce the variables
\begin{equation}
    w \coloneqq \frac{a^2}{6}+\frac{b^2}{2}
    \qquad \text{and} \qquad
    v_i \coloneqq H^{3-i}\,\widetilde{\ch}_i,
\end{equation}
which simplify the notation. In these coordinates, the $\nu_{a,b}$-slope for objects in $\Coh^{B_X}(X)$ becomes
\begin{equation*}
    \nu_{b,w}
    =
    \frac{v_2-wv_0}{v_1-bv_0}-b.
\end{equation*}
Thus, on the two-dimensional $(b,w)$-slice
\[
w \geq \frac{a(X,H)^2}{6} + \frac{b^2}{2},
\]
there is a wall-and-chamber decomposition for objects $E\in\Coh^{b}(X)$. The walls are line segments whose extensions all pass through the point
\[
\widetilde{\Pi}(E)
\coloneqq
\left(
\frac{H^2\widetilde{\ch}_1(E)}{H^3\widetilde{\ch}_0(E)},
\,
\frac{H\widetilde{\ch}_2(E)}{H^3\widetilde{\ch}_0(E)}
\right) = \left( \frac{v_1(E)}{v_0(E)}, \, \frac{v_2(E)}{v_0(E)}\right).
\]
In this notation, Corollary~\ref{Cor.BMT} can be reformulated as follows. Let $E\in\Coh^{B_X}(X)$ with $v_0(E)\neq 0$. If $E$ is $\nu_{b,w}$-semistable for
\begin{equation}\label{eq.bw condition}
    w=b^2-bv_1(E)+v_2(E),
\end{equation}
then
\begin{equation}\label{eq.bw bmt}
v_3(E)-bv_2(E)+\frac{b^2}{2}v_1(E)-\frac{b^3}{6}v_0(E)
\leq
3\left(w-\frac{b^2}{2}\right)\bigl(v_1(E)-bv_0(E)\bigr).    
\end{equation}

\begin{proof}[Proof of Theorem \ref{BMS}]
We have $e_1\geq 0$ since $E\in\Coh^{B_X}(X)$. If $e_1=0$, then the Bogomolov inequality \cite[Theorem~7.3.1]{Bayer-Macri-Toda:2014} implies that $e_0e_2\leq 0$. Therefore, the only remaining term in \eqref{eq:BMS}, namely
\[
\frac{4}{3}\left(-e_0e_2+\frac{a^2}{6}e_0^2\right)
\left(e_2-\frac{a^2}{6}e_0\right)^2,
\]
is non-negative, as claimed. Hence, we may assume that $e_1=v_1-bv_0>0$.

Now work in the $(b,w)$-coordinates. Consider the two parabolas
\[
w=\frac{b^2}{2}+\frac{a(X,H)^2}{6}
\qquad\text{and}\qquad
w=b^2-b\frac{v_1(E)}{v_0(E)}+\frac{v_2(E)}{v_0(E)},
\]
denoted by $\Gamma_1$ and $\Gamma_2$, respectively, and illustrated in Figure~\ref{fig:reflected polar}.

The two intersection points of $\Gamma_1$ and $\Gamma_2$ are precisely the tangent points of the two tangent lines from the pole $\widetilde{\Pi}(E)$ to the parabola $\Gamma_1$. Consequently, for any point $(b,w)$ lying above $\Gamma_1$, the line joining $(b,w)$ and $\widetilde{\Pi}(E)$ intersects $\Gamma_2$ at $\widetilde{\Pi}(E)$ and at a second point $(b',w')$, which also lies above $\Gamma_1$. A direct computation gives
\begin{equation}
    b'=\frac{v_2-v_0w}{v_1-v_0b}
    \qquad\text{and}\qquad
    w'=
    \left(\frac{v_2-v_0w}{v_1-v_0b}\right)^2
    -
    \frac{v_2b-v_1w}{v_1-v_0b}.
\end{equation}

By construction, $E$ is $\nu_{b',w'}$-semistable and satisfies \eqref{eq.bw condition}. Hence Corollary~\ref{Cor.BMT} yields \eqref{eq.bw bmt} at $(b', w')$. Finally, translating back to the original variables $(a,b)$ and $e_i$ gives the desired inequality \eqref{eq:BMS}.
\end{proof}

\begin{figure}
    \centering
    \begin{tikzpicture}
    \begin{axis} [
        axis lines = middle,         
        axis line style = {-triangle 45}, 
        xlabel = {$b, \frac{v_1}{v_0}$},
        ylabel = {$w, \frac{v_2}{v_0}$},
        xmin = -7.5, xmax = 7.5,      
        ymin = -3.0, ymax = 15.0,     
        ticks = none,                 
        width = 12cm,                 
        height = 8cm                  
    ]
        \addplot [domain=-6.928:6.928, samples=100] {0.25*x^2 + 1.5};
    
        \addplot [dashed, blue, domain=-6.5:7.2] {x + 4.5};
    
        \addplot [blue, domain=-4.568:6.568, samples=100] {0.5*x^2 - x - 1.5};
    
        \addplot [red, domain=-5.5:2.6] {-2*x + 2.5};
    
        \draw [red, dashed] (axis cs:2, -1.5) -- (axis cs:-7.0, 7.5);
        \draw [red, dashed] (axis cs:2, -1.5) -- (axis cs:7.0, 13.5);
        
        \node[circle, fill=black, inner sep=1.8pt, label= right:{$\widetilde{\Pi}(E)$}] 
            at (axis cs:2, -1.5) {};
    
    
        \node[circle, fill=black, inner sep=1.5pt, label=above right:{$(b', w')$}] 
            at (axis cs:-4, 10.5) {};
    
        \node[circle, fill=black, inner sep=1.5pt, label=above right:{$(b, w)$}] 
            at (axis cs:-2, 6.5) {};
    
    \end{axis}
    \end{tikzpicture}
    \caption{BMT(S) inequality 
    }
    \label{fig:reflected polar}
\end{figure}
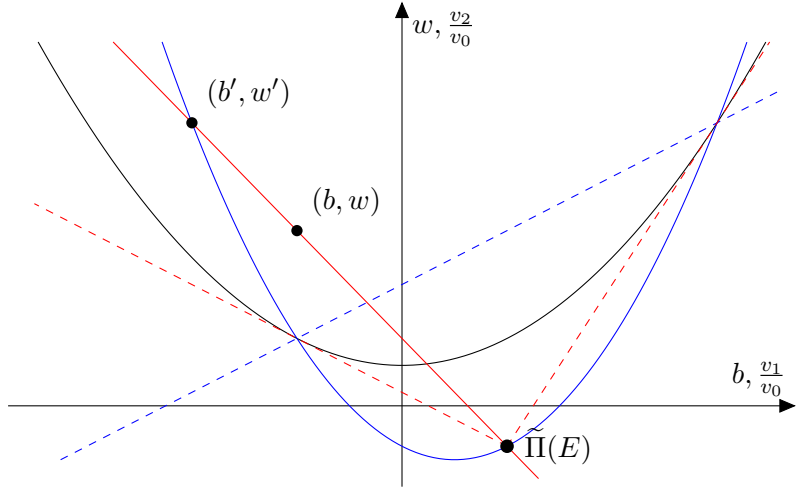

\bibliographystyle{alphaurl}
\bibliography{ref}

\end{document}